\documentclass[10pt]{amsart}
\usepackage[utf8]{inputenc}
\usepackage[T1]{fontenc}   
\usepackage{tcolorbox}
\usepackage{amssymb}
\usepackage{amsmath}
\usepackage{amsthm}
\usepackage{graphicx}
\usepackage{stmaryrd}
\usepackage{mathrsfs}
\usepackage{amscd}
\usepackage{bbm}
\usepackage{enumitem}
\usepackage{fncylab}
\usepackage{xcolor}
\usepackage{url}
\usepackage{leftidx}
\usepackage{calc}
\usepackage[width=\textwidth]{caption}
\usepackage[normalem]{ulem}
\usepackage{dsfont}

\theoremstyle{plain}
\newtheorem{theo}{Theorem}
\newtheorem{madef}{Definition}
\newtheorem{prop}{Proposition}

\newtheorem{lemma}{Lemma}

\newcommand{\un}{\mathbf{1}}
\def\pentsup#1{\left\lceil #1 \right\rceil}
\def\pentinf#1{\left\lfloor #1 \right\rfloor}

\renewcommand{\P}{\mathbb{P}}

\newcommand{\F}{\mathcal{F}}

\newcommand{\N}{\mathbb{N}}
\newcommand{\Z}{\mathbb{Z}}
\newcommand{\E}{\mathbb{E}}

\title[]{A probabilistic cellular automaton that admits no successful basic i.i.d. coupling}
\author{Jean B\'erard}
\address{Institut de Recherche Math\'ematique Avanc\'ee, UMR 7501, Universit\'e de Strasbourg et CNRS, 7 rue Ren\'e Descartes, 67\,000 Strasbourg, France}
\email{jberard@unistra.fr}

\begin{document}
\begin{abstract}
In this paper, we revisit a classic example of probabilistic cellular automaton (PCA) on $\{0,1\}^{\Z}$, namely, addition modulo 2 of the states of the left- and right-neighbouring cells, followed by either preserving the result of the addition, with probability $p$, or flipping it, with probability $1-p$. It is well-known that, for any value of $p \in ]0,1[$, this PCA is ergodic. We show that, for $p$ sufficiently close to $1$, no coupling of the PCA dynamics based on the composition of i.i.d. random functions of nearest-neighbour states (we call this a basic i.i.d. coupling), can be successful, where successful means that, for any given cell, the probability that every possible initial condition leads to the same state after $t$ time steps, goes to $1$ as $t$ goes to infinity. In particular, this precludes the possibility of a CFTP scheme being based on such a coupling. This property stands in sharp contrast with the case of monotone PCA, for which, as soon as ergodicity holds, there exists a successful basic i.i.d. coupling.   
\end{abstract}

\maketitle

\section{Introduction}

\subsection{PCA and coupling}
Probabilistic cellular automata (PCA) are discrete-time Markov processes formed by a lattice of stochastically evolving cells with local interaction (see e.g. \cite{TooVasStaMitKurPir,FerLouNar}). One of the key questions regarding their dynamical behaviour is that of ergodicity: does the distribution of the state of the PCA converge to a limit as time goes to infinity ? Among the tools that may be used to establish ergodicity is {\bf coupling}: loosely speaking, the idea is to define the various versions of the PCA starting from distinct initial configurations as a family of random variables on a common probability space, and then to show that, for large times, two such versions are, with high probability, close to each other. Obviously, the success of this approach relies on the ability to find a suitable coupling. Due to the local nature of the interaction, PCA naturally come with a family of couplings in which the update mechanism of the PCA is produced, at each time-step, by applying at each cell a random function of the current state of this cell and of its nearest neighbours, using an i.i.d. family of functions over different cells and time-steps. We call such a coupling a {\bf basic i.i.d. coupling}. Remarkably, it turns out that, for monotone PCA (see \cite{vdBSte}), a successful basic i.i.d. coupling exists as soon as the PCA is ergodic. This raises the question of whether a successful basic i.i.d. coupling always exists for ergodic PCA, regardless of their monotonicity. The present paper is devoted to showing, by means of a specific counterexample, that the answer is negative. Let us recall that, in \cite{Ber}, it is proved that sufficiently ergodic one-dimensional PCA always admit a locally defined successful coupling, albeit one that is very different from a basic i.i.d. coupling. The present counterexample shows that one must indeed look beyond basic i.i.d. couplings to get such general existence results.

\subsection{Main result}
We now turn to a precise statement of our result. Consider the function $f \   :   \   \{0,1\}^2 \to \{0,1\}$ defined by 
\begin{equation}\label{e:def-f}f(\alpha,\beta)=(\alpha+\beta) \mbox{ mod }2.\end{equation}
Then consider a one-dimensional nearest-neighbour probabilistic cellular automaton $(X_t)_{t \in \N}$ on $\{0,1\}^{\Z}$, with $X_t=(X_t(x))_{x \in \Z}$, whose Markov chain dynamics on $\{0,1\}^{\Z}$ is characterized as follows. 
\begin{itemize}
\item[\textbullet] For all $t \in \N$ and $x \in \Z$, the conditional distribution of $X_{t+1}(x)$ given $X_t$ satisfies
$$X_{t+1}(x)= \left\lbrace \begin{array}{ll}f(X_{t}(x-1), X_t(x+1)) & \mbox{ with probability }p \\ 1-f(X_{t}(x-1), X_t(x+1)) & \mbox{ with probability }1-p \end{array}\right. .$$
\item[\textbullet] For all $t \in \N$, the family of random variables $(X_{t+1}(x))_{x \in \Z}$ is conditionally independent given $X_t$.
 \end{itemize}
A local update function for the PCA dynamics is a random function $F \ : \ \{0,1\}^3 \to \{0,1\}$ such that 
\begin{equation}\label{e:update-fonc}  \forall (a,b,c) \in \{0,1\}^3, \   F(a,b,c) \sim p \cdot \delta_{f(a,c)} + (1-p) \cdot \delta_{1-f(a,c)}.\end{equation}
A basic i.i.d. coupling of the PCA dynamics is defined through an  i.i.d. family $(F_{x,t})_{x \in \Z, t \in \N}$ of local update functions, and the explicit equation: \begin{equation}\label{e:flot} X_{t+1}(x) = F_{x,t}(X_{t}(x-1), X_t(x), X_t(x+1)).\end{equation}
Starting from an initial configuration $X_0=\xi \in \{0,1\}^{\Z}$ at time $0$, we denote by $X_t^{\xi}$ the (random) configuration at time $t$ resulting from iteratively applying \eqref{e:flot} from time $0$ to time $t$.
Our main result is the following. 
 \begin{theo}\label{t:letheo}
For all $p$ close enough to $1$, every coupling characterized by \eqref{e:update-fonc} and \eqref{e:flot} is such that
$$\limsup_{t \to +\infty} \P \left( \exists \ \xi_1, \xi_2 \in \{0,1\}^{\Z} \mbox{ such that } X_t^{\xi_1}(0) \neq X_t^{\xi_2}(0) \right) > 0.$$ 
\end{theo}
Theorem \ref{t:letheo} shows in particular that coupling from the past (CFTP) based on a basic i.i.d. coupling is not feasible for $p$ close to $1$: indeed, for CFTP based on i.i.d. copies of $F$ to be successful, we would need to have 
$$\lim_{t \to +\infty} \P \left( \exists \ \xi_1, \xi_2 \in \{0,1\}^{\Z} \mbox{ such that } X_t^{\xi_1}(0) \neq X_t^{\xi_2}(0) \right) = 0.$$

\subsection{Existing results}

The PCA we study is a classic example of an ergodic but non-monotonic PCA, see \cite{TooVasStaMitKurPir} (Example 1.3, p. 11, referring back to \cite{Vas}), see also \cite{MarSabTaa} (Section 5.1). Here, our goal is not to study the dynamics of the PCA itself, but rather the dynamics of discrepancies between two versions of the PCA, starting from distinct initial configurations, and using an otherwise arbitrary basic i.i.d. coupling. The stochastic process tracking these discrepancies -- called in the sequel the flip/preserve process -- can be described as another PCA within a random environment, the random environment itself evolving as the original PCA. The resulting flip/preserve process turns out to be related (although not identical) to the PCA discussed in \cite{DomKin,Kin} and in \cite{Dur2} (p. 115), with a full mathematical treatment given in \cite{BraNeu} ("rule 90"). Our approach for proving Theorem \ref{t:letheo} consists in adapting the arguments developed in \cite{BraNeu}, exploiting the mixing properties of the random environment.

\subsection{Organization of the paper}

In Section \ref{s:flip-preserve}, we define the flip/preserve process, and establish some of its basic properties. In Section \ref{s:param-list}, we give the definition of the various parameters involved in the ensuing sections.
In Section \ref{s:mixing}, we state (and prove) certain mixing properties of the random environment, then show how these enable us to ensure that the environment is "nice" with high probability. In Section \ref{s:dynamics-flip-preserve}, the two key ingredients used in \cite{BraNeu} are adapted to the setting of the present flip/preserve process: controlling the occurrence of large intervals of $0$, and the dynamics of the rightmost $1$. Finally, in Section \ref{s:preuve-theo}, we put together the results of the previous sections in order to prove Theorem \ref{t:letheo}, using comparison with oriented percolation in a way similar to \cite{BraNeu}.

\subsection{Notation for intervals}

Given $u,v \in \Z$ such that $u \leq v$, we let $\llbracket u,v \rrbracket = \{ z \in \Z ; \ u \leq z \leq v \}$. To avoid confusion, we distinguish the {\it size} (= number of elements) of $\llbracket u,v \rrbracket$, which is equal to $v-u+1$, and its {\it length}, which is equal to $v-u$.

\section{The flip/preserve process}\label{s:flip-preserve}

\subsection{Definition and first properties}

We define the flip/preserve process $(Y_t)_{t \in \N}$, starting from $\zeta \in \{0,1\}^{\Z}$, by letting
 \begin{equation}\label{e:flip-preserve}Y^{\xi, \zeta}_t(x) = \left( X_t^{\xi+\zeta}(x)  - X_t^{\xi}(x) \right) \mbox{ mod 2}.\end{equation}
 The flip/preserve process tracks the discrepancies between $X_t^{\xi+\zeta}$ and $X_t^{\xi}$. In fact, \eqref{e:flip-preserve} may be rewritten as $Y^{\xi,\zeta}(x) = \un\left(X_t^{\xi+\zeta}(x)  \neq X_t^{\xi}(x) \right)$.
 Note that the value of $X_t^{\xi+\zeta}$ can be recovered from the values of $X_t^{\xi}$ and $Y^{\xi,\zeta}_t$, through the identity \begin{equation}\label{e:ricovert}X_t^{\xi+\zeta}(x)  = \left(   X_t^{\xi}(x) + Y^{\xi,\zeta}_t(x)  \right) \mbox{ mod 2}.\end{equation}

To alleviate notations a bit, let us introduce the random map\footnote{Measurability is not an issue here, since we limit ourselves to maps $\psi \ : \ \{0,1 \}^{\Z} \to \{0,1 \}^{\Z}$ such that $\psi[s](x)$ is a function of $s(x-1),s(x),s(x+1)$, so that $\psi$ can be viewed as a countable collection of elements of the (finite) set of maps from $\{0,1\}^3$ to $\{0,1\}$.} $F_t \ : \ \{0,1 \}^{\Z} \to \{0,1 \}^{\Z}$ defined by   
\begin{equation}\label{e:flot-concis} F_t[\eta](x) =F_{x,t}(\eta(x-1), \eta(x), \eta(x+1)),\end{equation}
where $\eta \in  \{0,1 \}^{\Z}$, $x \in \Z$, and where the random maps $F_{x,t}$ are the local update functions appearing in \eqref{e:flot}.   
By \eqref{e:flot-concis}, we have that $X^{\xi}_{t+1}=F_t[X^{\xi}_t]$ and $X^{\xi+\zeta}_{t+1}=F_t[X^{\xi+\zeta}_t]$, so that, using \eqref{e:ricovert} and \eqref{e:flot-concis}, we can write  
\begin{equation}\label{e:recure-structure}\begin{pmatrix} X^{\xi}_{t+1} \\ Y^{\xi,\zeta}_{t+1} \end{pmatrix} = \begin{pmatrix} F_t[X^{\xi}_t] \\ \left(F_t [ X^{\xi}_t+Y^{\xi,\zeta}_t  \mbox{ mod 2}] - F_t[X^{\xi}_t] \right) \mbox{ mod 2}.\end{pmatrix}\end{equation}
From \eqref{e:recure-structure} and the fact that the sequence $(F_t)_{t \in \N}$ is i.i.d., it is apparent that the sequence $(X_t^{\xi},Y_t^{\xi,\zeta})_{t \in \N}$ has a time-homogeneous Markov chain structure on $\{0,1\}^{\Z} \times \{0,1\}^{\Z}$.  Since we are dealing with Markov processes, we use the classical notations $\P_{\xi}$ and $\P_{\xi,\zeta}$ to refer to the distributions of $(X_t^{\xi})_{t}$ and $(Y^{\xi,\zeta}_t)_t$,  and simply write 
$\P_{\xi}((X_t)_t \in \cdot)$ instead of $\P \left( (X_t^{\xi})_t \in \cdot \right)$, and similarly $\P_{\xi,\zeta}((X_t, Y_t)_t \in \cdot)$ instead of $\P \left((X_t^{\xi}, Y_t^{\xi,\zeta})_t \in \cdot \right)$.

To give a more explicit description of the corresponding Markov structure, we now define the transition kernel $q : \{0,1\}^3 \times \{0,1\}^3 \to [0,1]$ by 
$$q[(a,b,c),(u,v,w)]=\P\left[ F(a+u,b+v,c+w) = F(a,b,c) \right],$$
where $F$ is a local update function as in \eqref{e:update-fonc}. From \eqref{e:recure-structure}, one checks that, with respect to the conditional probability $\P_{\xi,\zeta}(\cdot | X_t,Y_t)$, we have that\footnote{Strictly speaking, these properties are true a.s. and we have to consider a regular version of the conditional probability, whose existence here is completely standard.}: 
\begin{itemize}
\item[\textbullet] $Y_{t+1}(x) = \left\lbrace \begin{array}{ll}0 & \mbox{ with probability }q[\mbox{\small $X_{t}( x-1,x,x+1), Y_{t}(x-1,x,x+1)$}]\\ 1 & \mbox{ with probability }1- q[\mbox{\small $X_{t}(x-1,x,x+1), Y_{t}(x-1,x,x+1)$}] \end{array}\right.,$
\item[\textbullet] $(Y_{t+1}(x))_{x \in \Z}$ is independent.
\end{itemize}

In this sense, we may think of $(Y_t)_{t \in \N}$ as a probabilistic cellular automaton on $\{0,1\}^{\Z}$ evolving within a random environment provided by the dynamics of $(X_t)_{t \in \N}$. 


 We now list two important properties of the kernel $q$. The first one is an immediate consequence of the definition of $q$.
  \begin{equation}\label{e:zero-zero}
 \forall (a,b,c) \in \{0,1\}^3 ,  \  q[(a,b,c),(0,0,0)]=1.
 \end{equation}
 The second one\footnote{Proof: Let $A$ and $B$ be Bernoulli($r$) random variables.  Then
 $\P(A \neq B)=\P(A=0,B=1)+\P(A=1,B=0)$, so $\P(A \neq B) \leq \P(A=0)+\P(B=0) = 2(1-r)$, and $\P(A \neq B) \leq \P(A=1)+\P(B=1) = 2r$, so that $\P(A \neq B) \leq \min(2r, 2(1-r))$. Now, 
if $f(u,w)=0$, both $F(a+u,b+v,c+w)$ and $F(a,b,c)$ have the same distribution, wich is either Bernoulli$(p)$ or Bernoulli$(1-p)$. Hence we can write $ \left| \vphantom{\sum} q[(a,b,c),(u,v,w)]-(1-f(u,w)) \right| = \left| \vphantom{\sum} \P(A=B)-1\right| = \P(A \neq B)$, with $r=p$ or $r=1-p$, and we are done. If $f(u,w)=1$, either  $F(a+u,b+v,c+w)$ has a Bernoulli$(p)$ distribution and $F(a,b,c)$ has a Bernoulli$(1-p)$ distribution, or the other way round. Hence we can write $ \left| \vphantom{\sum} q[(a,b,c),(u,v,w)]-(1-f(u,w)) \right|=\P(A = 1-B)=\P(A \neq B)$, with $r=p$ or $r=1-p$, and again we are done.} precisely quantifies the fact that, for $p$ close to $1$, the dynamics of $(Y_t)$ is close to being deterministic. 
  \begin{equation}\label{e:perturb-determ}
 \forall (a,b,c),(u,v,w) \in \{0,1\}^3, \  \left| \vphantom{\sum} q[(a,b,c),(u,v,w)]-(1-f(u,w)) \right| \leq 2(1-p),
  \end{equation}
where $f$ is defined in \eqref{e:def-f}.
We now let  
\begin{equation}\label{d:epsilon}\varepsilon=\max_{(a,b,c)} \max(q[(a,b,c),(0,0,1)], q[(a,b,c),(1,0,0)])\end{equation}
 and let $(a^*,b^*,c^*)$ be such that 
 \begin{equation}\label{d:abc-etoile} \max(q[(a^*,b^*,c^*),(0,0,1)], q[(a^*,b^*,c^*),(1,0,0)])=\varepsilon.\end{equation}
Note that, by \eqref{e:perturb-determ}, we have the inequality
\begin{equation}\label{e:epsilon-petit}\varepsilon \leq 2(1-p).\end{equation}
 
When $\varepsilon=0$, we briefly show in the next section that coupling never occurs. Thus, except in the next section, {\bf we always assume that $\varepsilon > 0$}.
In the sequel, $\varepsilon$ is the key "small" parameter on which the subsequent constructions are based. Note that, while $p$ is a parameter of the original PCA, $\varepsilon$ is a parameter of the basic i.i.d. coupling under consideration.

\subsection{The degenerate case $\varepsilon=0$}\label{s:degenere}

Assume that $\varepsilon=0$. Consider any $\xi \in \{0,1\}^{\Z}$, and, given $n \geq 1$, let $\zeta(n)=1$, and $\zeta(x)=0$ for any $x \neq n$.
Due to the fact that $\varepsilon=0$, we have that $q[(a,b,c),(0,0,1)]=0$ for any triple $(a,b,c)$. Combining this property with \eqref{e:zero-zero}, an immediate induction shows that, for all $0 \leq t \leq n$,  
$Y_t^{\xi,\zeta}(n-t)=1$ a.s.    As a consequence, we have that $Y_n^{\xi,\zeta}(0)=1$ a.s., so that $X_n^{\xi}(0) \neq X_n^{\xi+\zeta}(0)$ a.s. We deduce that, for all $n \geq 1$, 
 $\P \left( \exists \ \xi_1, \xi_2 \in \{0,1\}^{\Z} \mbox{ such that } X_n^{\xi_1}(0) \neq X_n^{\xi_2}(0) \right)=1$.

 \subsection{Antecedents for the deterministic dynamics}

 Define a map $\phi \ : \ \{0,1\}^{\Z} \to \{0,1\}^{\Z}$ by letting, for every $\zeta \in \{0,1\}^{\Z}$: 
 \begin{equation}\label{e:deterministic-CA}\phi[\zeta](x)=f(\zeta(x-1),\zeta(x+1)).\end{equation}
 The map $\phi$ defines a deterministic cellular automaton, which coincides with the dynamics of $(X_t)$ and of $(Y_t)$ when $p=1$.

 We say that a sequence $\zeta' \in \{0,1\}^{\Z}$ is a deterministic antecedent of $\zeta$ when $\phi[\zeta']=\zeta$. We denote by $\phi^{-1}[\zeta]$ the set of deterministic antecedents of $\zeta$.

The possible deterministic antecedents of the null sequence, and other, related ones, are listed in Table \ref{tab:antecedents} (where our notations for sequences are explained in detail). The proofs stem from the definition of the deterministic dynamics and explicit enumeration of possible cases.
\renewcommand{\arraystretch}{1.5}
 \begin{table}
 \centering
 \begin{tabular}{|c|c|}
 \hline
 sequence & deterministic antecedents \\ \hline
 $\overline{\d{0}}$ & $\overline{\d{0}}$, $\overline{\d{1}}$, $\overline{\d{0}\d{1}}$ \\  \hline
  $\overline{\d{1}}$ & $\overline{\d{1}\d{1}\d{0}\d{0}}$ \\  \hline
  $\overline{\d{0}\mbox{1}}$ & $\overline{\d{1}\mbox{0}\d{0}\mbox{0}}$, $\overline{\d{0}\mbox{1}\d{1}\mbox{1}}$ \\  \hline
  $\overline{\d{0}\mbox{111}}$ & $\overline{0\d{0}\mbox{011}\d{0}\mbox{11}}$, $\overline{\mbox{1}\d{1}\mbox{100}\d{1}\mbox{00}}$ \\
  \hline
  \end{tabular}
  \vspace{1em}
 \caption{Given a finite $\{0,1\}-$valued sequence $a_1, \ldots , a_{\ell},\ldots, a_m$  we use the notation $\overline{a_1 \cdots \d{$a_{\ell}$} \cdots a_{m}}$ for the periodic sequence  $\xi \in \{0,1\}^{\Z}$ defined by
 $\xi(x)=a_{(x+\ell) \mbox{ mod }m}$. When several underdots are used, the notation stands for the set of sequences obtained for the various underdot positions, e.g. $\overline{\d{$a$}b\d{$c$}}= \{  \overline{\d{$a$}bc}, \overline{ab\d{$c$}} \}$. We also write $\overline{\d{$a$}}$ instead of the clumsier $\{ \overline{\d{$a$}} \}$ when it is clear from the context that we are dealing with sets of sequences.} 
 \label{tab:antecedents}
 \end{table}
  Denote by $\mathcal{P}= \overline{\d{0}} \cup \overline{\d{1}} \cup   \overline{\d{0}\mbox{1}} \cup \overline{\d{0}\mbox{1}\mbox{1}\mbox{1}}$ the set of sequences for which the deterministic antecedents are shown in Table \ref{tab:antecedents}, and denote by $\mathcal{P'}=  \overline{\d{0}}  \cup \overline{\d{1}}  \cup \overline{\d{0}\d{1}}  \cup\overline{\d{0}\d{1}\d{1}\d{1}}$ the set formed by the various translations of sequences in $\mathcal{P}$.
 Note that, since $\phi$ commutes with space translations, the antecedents of sequences in $\mathcal{P'}$ 
 are translations of the antecedents shown in the table. Given a finite (non-empty) interval $\llbracket u,v\rrbracket$, and a sequence $\zeta \in \mathcal{P'}$, a sequence $\sigma \in \{0,1\}^{\Z}$ is such that $\phi[\sigma]$ coincides with $\zeta$ on $\llbracket u,v \rrbracket$ if and only if $\sigma$ coincides on $\llbracket u-1,v+1 \rrbracket$ with one of the four sequences in $\phi^{-1}[\zeta]$.

 \section{List of parameters}\label{s:param-list}

 The construction leading to the proof of Theorem \ref{t:letheo} involves the choice of a certain number of parameters, which are listed below. In the sequel, for the sake of clarity, we try as much as possible to obtain general estimates involving these parameters, and only at the very end consider their specific behaviour when $p \to 1$. 
  
 \begin{itemize}
 \item[\textbullet] $K = 10 \pentsup{\varepsilon^{-3/2}}$ 
 \item[\textbullet] $T=4K$
 \item[\textbullet] $\ell_0=3K/10$
 \item[\textbullet] $\ell_n= \pentsup{\ell_0 \cdot\log(1/\varepsilon)^{-2n}}$, where $n \geq 1$ (we only use $n=1,2,3$)
 \item[\textbullet] $\lambda=\pentsup{\varepsilon^{-3/4}}$
 \item[\textbullet] $\theta=e^{-\varepsilon^{-1/3}}$
 \item[\textbullet] $I=\llbracket -K, +K \rrbracket$, 
 \item[\textbullet] $J_{+}= \llbracket +6K/5, +9K/5 \rrbracket$
  \item[\textbullet] $J^1_{+}= \llbracket +6K/5, +7K/5 \rrbracket$
 \item[\textbullet] $J^2_{+}= \llbracket +7K/5, +8K/5 \rrbracket$
 \item[\textbullet] $J^3_{+}= \llbracket +8K/5, +9K/5 \rrbracket$
 \item[\textbullet] $J_{-}=\llbracket -9K/5, -6K/5 \rrbracket$
 \item[\textbullet] $J=J_- \cup J_+$
  \end{itemize}

 Note that the parameters above do not involve the initial conditions of the PCA, and depend on the coupling (i.e. on the distribution of the update function $F$) only through the value of $\varepsilon$. Thanks to \eqref{e:epsilon-petit}, any combination of these parameters going to $+\infty$ (or $0$) as $\epsilon \to 0$ will do so as $p \to 1$, uniformly with respect to the initial conditions and to $F$.   
  
To ensure that the sequence $(\ell_n)_{n \geq 0}$ is non-increasing, one may assume from now on that $ p \geq 1-1/(2e)$.

\section{Mixing properties of the PCA $(X_t)_{t \geq 0}$}\label{s:mixing}

In this section, we give a short proof of the fact that the limiting distribution of the PCA $(X_t)_{t \geq 0}$ is the product Bernoulli measure $\mu=\mbox{Bernoulli}(1/2)^{\otimes \Z}$ on $\{0,1\}^{\Z}$, with an explicit control upon the convergence speed. We then deduce several estimates on the probability that $X_t$ is a "nice" sequence, in a specific sense (see Section \ref{ss:niceties}).

\subsection{Convergence speed}

For $A \subset \Z$, we denote by $\pi_A$ the canonical projection from $\{0,1\}^{\Z}$ to $\{0,1\}^{A}$.

\begin{prop}\label{p:melange-ACP}
For any initial condition $X_0=\xi \in \{0,1\}^{\Z}$, one has the following bound, valid for every non-empty subset $A \subset \Z$ and for all $t \geq 1$:
$$d_{\mbox{\normalfont \scriptsize TV}}\left( \mbox{\normalfont Law}(\pi_A(X_t)),\pi_A(\mu) \right) \leq |A| \cdot d_t, \mbox{ where }d_t=\frac{1}{2}|2p-1|^t.$$
\end{prop}

\begin{lemma}\label{l:melange-bernou}
Let $W_1,\ldots,W_t$ be i.i.d.  Bernoulli$(1-p)$ random variables.
Then the total variation distance between the distribution of $W_1+\cdots+W_t \mbox{ mod }2$ and the Bernoulli$(1/2)$ distribution is $\frac{1}{2}|2p-1|^t$.
\end{lemma}
\begin{proof}[Proof of Lemma \ref{l:melange-bernou}]
Let $V=W_1+\ldots+W_t$, so that $V$ follows the binomial $(t,(1-p))$ distribution. As a consequence, $\P(V \mbox{ mod }2 =1)=\sum_{k \in \llbracket 0,t\rrbracket \atop \mbox{\scriptsize odd $k$} } {t \choose k} (1-p)^k p^{t-k}$ and
 $\P(V \mbox{ mod }2 =0)=\sum_{k \in \llbracket 0,t\rrbracket \atop \mbox{\scriptsize even $k$} } {t \choose k} (1-p)^k p^{t-k}$, so that 
 $\P(V \mbox{ mod }2 =0)-\P(V \mbox{ mod }2 =1) = (-(1-p)+p)^t$. Since $\P(V \mbox{ mod }2 =0)+\P(V \mbox{ mod }2 =1)=1$, we get the values
 $\P(V \mbox{ mod }2 =0)=1/2+(1/2) (-(1-p)+p)^t$ and $\P(V \mbox{ mod }2 =1)=1/2-(1/2) (-(1-p)+p)^t$.
\end{proof}

\begin{proof}[Proof of Proposition \ref{p:melange-ACP}]
For the proof, we introduce a version of $X_t$, denoted by $X'_t$, which does not rely on \eqref{e:flot}. Specifically, given an i.i.d. family of Bernoulli$(1-p)$ random variables $(W_{x,t})_{x \in \Z, t \geq 1}$, we start from the initial configuration $X'_0 =\xi$, and set, for all $t \geq 0$: 
\begin{equation}\label{e:iter-coupl-add}X'_{t+1}(x)= \left( X'_{t}(x-1)+X'_t(x+1)+W_{x,t+1} \right) \mbox{ mod }2.\end{equation}
For a given initial condition $X'_0=X_0=\xi$, and any fixed $t$, the distributions of $X_t$ and $X'_t$ coincide, so that it is enough to prove the conclusion of Proposition \ref{p:melange-ACP} for $X'_t$ instead of $X_t$.

Consider $t \geq 1$ and a subset $A \subset \Z$ with $N \geq 1$ elements, and write $x_1 < \cdots < x_N$ the elements of $A$. Iterating \eqref{e:iter-coupl-add}, we see that, for $1 \leq n \leq N$, $X'_{t}(x_n)$ can be written as a sum 
$$X'_{t}(x_n)= \left( \sum_{x=x_n-t}^{x_n+t} \gamma^{x_n}_{x,t} \cdot X'_0(x) + \sum_{k=1}^t \sum_{x=x_n-k+1}^{x_n+k-1} \delta^{x_n}_{x,k} \cdot W_{x,t-k+1} \right)  \mbox{ mod }2,$$
where $\gamma^{x_n}_{x,t}$ and $\delta^{x_n}_{x,k}$ are (non-random) integer numbers, and with $\delta^{x_n}_{x_n+k-1,k}=1$ for $k=1,\ldots, t$.
As a consequence, we can rewrite $X'_t(x_n)$ as 
\begin{equation}\label{e:regroupement}X'_{t}(x_n)=U_{n-1} \mbox{ mod }2 + \left(  \sum_{k=1}^t W_{x_n+k-1,t-k+1} \right) \mbox{ mod }2,\end{equation}
 where  $U_{n-1}$ is a (measurable) function of $X'_0(x)$ for $x_n-t \leq x \leq x_n+t$, and of $W_{x,t-k+1}$ for $1 \leq k \leq t$ and $x_n-k+1 \leq x < x_n+k-1$.

Starting with the case $n=1$, we observe that $S=\sum_{k=1}^t W_{x_1+k-1,t-k+1}$ is independent from $U_0$, so that, in view of \eqref{e:regroupement}, the distribution of $X'_t(x_1)=U_0+S \mbox{ mod 2}$ is a mixture of the distribution $\nu_1$ of  $S \mbox{ mod 2}$, and of the distribution $\nu_2$ of $1+S \mbox{ mod 2}$. Letting $\nu$ denote the Bernoulli$(1/2)$ distribution, we deduce from Lemma \ref{l:melange-bernou} that  $d_{\mbox{\normalfont \scriptsize TV}}(\nu_1, \nu) \leq  \frac{1}{2}|2p-1|^t$, and similarly that  $d_{\mbox{\normalfont \scriptsize TV}}(\nu_2, \nu) \leq  \frac{1}{2}|2p-1|^t$. Using the bound $d_{\mbox{\normalfont \scriptsize TV}}(\alpha \nu_1 + (1-\alpha) \nu_2, \nu) \leq \alpha d_{\mbox{\normalfont \scriptsize TV}}(\nu_1, \nu)+ (1-\alpha) d_{\mbox{\normalfont \scriptsize TV}}(\nu_2, \nu)$, we deduce that the total variation distance between the distribution of $X_t(x_1)$ and the 
Bernoulli$(1/2)$ distribution is $\leq  \frac{1}{2}|2p-1|^t$.

Then, using  $\eqref{e:regroupement}$ and the fact that $x_1<\cdots<x_N$ , we deduce that, for $n=2,\ldots, N$, 
$\sum_{k=1}^t W_{x_n+k-1,t-k+1}$ is independent from $X'_t(x_1),\ldots, X'_{t}(x_{n-1})$, $U_{n-1}$. Using Lemma \ref{l:melange-bernou} again, we deduce that the total variation between the conditional distribution of $X'_t(x_n)$ given $X'_t(x_1),\ldots, X'_{t}(x_{n-1})$, and the Bernoulli$(1/2)$ distribution, is $\leq \frac{1}{2}|2p-1|^t$.

We conclude\footnote{Here we repeatedly use the elementary result that, given two pairs of random variables $(\Gamma_1,\Gamma_2)$ and $(\Delta_1,\Delta_2)$ taking values in a finite set, $d_{\mbox{\normalfont \scriptsize TV}}( \mbox{Law}(\Gamma_2),  \mbox{Law}(\Delta_2)   ) \leq d_{\mbox{\normalfont \scriptsize TV}}( \mbox{Law}(\Gamma_1),  \mbox{Law}(\Delta_1)   ) + \max_{u}d_{\mbox{\normalfont \scriptsize TV}}( \mbox{Law}(\Gamma_2 | \Gamma_1=u),  \mbox{Law}(\Delta_2|\Delta_1=u)   ) $, with the convention that $d_{\mbox{\normalfont \scriptsize TV}}( \mbox{Law}(\Gamma_2 | \Gamma_1=u),  \mbox{Law}(\Delta_2|\Delta_1=u))=0$ if any of the events $\{\Gamma_1=u\}$ or $\{\Delta_1=u\}$ has probability zero.} that the total variation between the (joint) distribution of $\left(X'_t(x_1),\ldots, X'_{t}(x_{N})\right)$ and the $\mbox{Bernoulli}(1/2)^{\otimes N}$ distribution, is bounded above by $N \cdot \frac{1}{2}|2p-1|^t$.

\end{proof}

\subsection{Nice sequences}\label{ss:niceties}

\begin{madef}
 Given an interval $\llbracket g,h\rrbracket \subset \Z$ of size $\geq 10$, we say that a $\{0,1\}-$valued sequence $(s_x)_{x \in \llbracket g,h\rrbracket}$ is nice if, for every $i \in \{1,\ldots, 8\}$,
\begin{equation}\label{e:jojo}\frac{\left| \{x \in  \llbracket g, h-2 \rrbracket \cap (i+8 \Z) \mbox{ such that }(s_{x} ,s_{x+1}, s_{x+2} )= (a^*, b^*, c^*) \} \right|}{ \left| \llbracket g, h-2 \rrbracket \cap (i+8 \Z) \right|} \geq 1/9.\end{equation}
\end{madef}
\begin{madef}
Given a finite subset $A \subset \Z$, and $10 \leq \ell \leq |A|$, we say that a sequence $s \in \{0,1\}^A$  is $\ell-$nice if $s(\llbracket g,h\rrbracket)$ is nice for every interval $\llbracket g,h\rrbracket \subset A$ of size $\geq \ell$.
\end{madef}

For the set $J$ defined in Section \ref{s:param-list}, we have the following result.
\begin{lemma}\label{l:nice}
If $X_0 \sim \mu$, and $10 \leq \ell \leq |J|$, the probability that $X_0(J)$ is not $\ell-$nice is bounded above by \begin{equation}\label{e:eK}e_{K,\ell}=c |J|^2 e^{-\rho \ell},\end{equation}
where $\rho$ and $c$ are positive constants. 
\end{lemma}
\begin{proof}
Given an interval $\llbracket g,h\rrbracket \subset J$ of size $m \geq \ell$, and $i \in \{1,\ldots, 8\}$, we have that $ \llbracket g, h-2 \rrbracket \cap (i+8 \Z)$ contains at least   $\pentinf{\frac{m-2}{8}}$ elements.  
Since any two distinct elements of $ \llbracket g, h-2 \rrbracket \cap (i+8 \Z)$ are separated by at least $8$ units, $(X_0(x),X_0(x+1),X_0(x+2))_{x \in \llbracket g, h-2 \rrbracket \cap (i+8 \Z)}$ forms an i.i.d. family of triples, whose common distribution satisfies
 $\P((X_0(x),X_0(x+1),X_0(x+2))=(a^*, b^*, c^*))=1/8$. The probability that, among a family of at least $\pentinf{\frac{m-2}{8}}$ such i.i.d. triples, the proportion of those equal to 
$(a^*, b^*, c^*)$ falls below $1/9$, is bounded above by $c_1 e^{-\rho m}$, for suitable positive constants $\rho,c_1$, using a standard large deviations bound. Given $m$, there are at most $|J|$ intervals $\llbracket g,h\rrbracket \subset J$ of size $m$. Moreover, one has $m \leq |J|$, so we have to deal with at most $|J|^2$ intervals. Using the fact that $e^{-\rho m} \leq e^{-\rho \ell}$ since $m \geq \ell$, and summing over all possible cases, the probability for $X_0(J)$ not to be $\ell-$nice is bounded above by $8 |J|^2 c_1 e^{-\rho \ell}$.
\end{proof}

Using Proposition \ref{p:melange-ACP}, we deduce that, for any initial configuration $X_0=\xi$ at time $0$, we have that 
$$\P_{\xi}\left(X_t(J) \mbox{ is not $\ell-$nice}\right) \leq  e_{K,\ell} + |J| d_t .$$
By the union bound, and the fact that the sequence $(d_t)_{t \geq 0}$ is non-increasing, we then deduce that 
\begin{equation}\label{e:reunion-lignes}\P_{\xi}\left( \bigcup_{t=K}^{2K-1} \{ X_t(J) \mbox{ is not $\ell-$nice} \}\right) \leq  K (e_{K,\ell} + |J|d_K).\end{equation}
Define $$Q_{\ell}(\xi)=\P_{\xi} \left(  \bigcup_{t=0}^{K-1} \{ X_t(J) \mbox{ is not $\ell-$nice} \}  \right).$$
Note that the definition of $Q_{\ell}$ involves the dynamics of $X_t(J)$ from $t=0$ to $t=K-1$, and also that $J \subset \llbracket -2K, +2K \rrbracket$. Due to the nearest-neighbour 
character of the PCA dynamics, we deduce that
\begin{equation}\label{e:Q-local}\mbox{ $Q_{\ell}(\xi)$ is a function of $\xi(\llbracket -3K,+3K \rrbracket)$ only.}\end{equation}
 By the Markov property, 
we have that 
$$ \P_{\xi}\left( \bigcup_{t=K}^{2K-1} \{ X_t(J) \mbox{ is not $\ell-$nice} \}\right) = \E_{\xi}(Q_{\ell}(X_K)),$$
so \eqref{e:reunion-lignes} implies that
$$\E_{\xi}(Q_{\ell}(X_K)) \leq  K (e_{K,\ell} + |J|d_K).$$
By Markov's inequality, 
\begin{equation}\label{e:merci-Markov}\P_{\xi} \left( Q_{\ell}(X_K) > \theta  \right) \leq  \theta^{-1} K (e_{K,\ell} + |J|d_K).\end{equation}
We now observe that, if $\xi$ is such that $Q_{\ell}(\xi) \leq \theta$, 
\begin{equation}\label{e:controle-nice}\P_{\xi}\left( \bigcup_{t=0}^{4K-1} \{ X_t(J) \mbox{ is not $\ell-$nice}\}  \right) \leq \theta + 3 K (e_{K,\ell} + |J|d_K),\end{equation}
using the definition of $Q_{\ell}$ for $t$ between $0$ and $K-1$, then conditioning by $X_{(i-1)K}$ and using \eqref{e:reunion-lignes} for $t$ between $iK$ and $(i+1)K-1$, with $i=1,2,3$. 
Similarly, using translation invariance, we have that 
\begin{equation}\label{e:controle-nice-2}\P_{\xi}\left( \bigcup_{t=4K-2}^{4K-1} \{X_t(J+2K)\mbox{ is not $\ell-$nice} \}\right) \leq  2(e_{K,\ell} + |J|d_K),\end{equation}
Moreover, conditioning by $X_{3K}$ and using \eqref{e:merci-Markov} and translation invariance, we also have that 
\begin{equation}\label{e:controle-nice-3}\P_{\xi}\left[ \vphantom{\sum} Q_{\ell}\left((X_{4K}(x+2K))_{x \in \Z} \right) > \theta \right] \leq  \theta^{-1} K (e_{K,\ell} + |J|d_K).\end{equation}

\section{Dynamics of the flip/preserve process}\label{s:dynamics-flip-preserve}

In this section, we adapt the arguments developed in \cite{BraNeu} to control the occurrence of large intervals of $0$s, and the dynamics of the rightmost $1$, in the flip/preserve process. The present situation differs from \cite{BraNeu}, since, in our setting, the key parameter $\varepsilon$ is only involved in probabilities of transitions $(0,0,1) \to 0/1$ or $(1,0,0) \to 0/1$, and provides the exact value of the transition probability only when the underlying environment is $(a^*,b^*,c^*)$. As regards the other transition probabilities, beyond the fact that a transition from $(0,0,0)$ always leads to a $0$, the only control we have is the $2(1-p)$ bound \eqref{e:perturb-determ} on the difference with respect to the deterministic dynamics. 

\subsection{Controlling the occurrence of a large interval of $0$s}

Broadly speaking, the idea is that, if a large interval of $0$s appears at time $t$, where at least a $1$ was present at time $t-1$, either the configuration at time $t-1$ coincides with a deterministic antecedent of $\overline{0}$ on a smaller (but still large) interval, or the deterministic dynamics is not followed from time $t-1$ to time $t$ for a large number of sites. Iterating such an argument up to three times leads to deterministic antecedents containing large numbers of $001$ and $100$, for which the definition of $\varepsilon$ allows us to get the control we need.

\begin{lemma}\label{l:plage-de-zeros}
Assume that $\zeta(\llbracket -\ell_0,+\ell_0 \rrbracket) \neq (0,\ldots,0)$, and that there exist $x \in \llbracket \ell_0+2, 2\ell_0+1 \rrbracket$ such that $\zeta(\llbracket x-1,x \rrbracket)=(0,0)$ and $x' \in \llbracket -2\ell_0-1, -\ell_0-2 \rrbracket$ such that $\zeta(\llbracket x',x'+1 \rrbracket)=(0,0)$, where $\ell_0 \geq 1$. Then, for every $\xi$, $$\P_{\xi,\zeta}(Y_1(\llbracket -2\ell_0,+2 \ell_0 \rrbracket)=(0,\ldots,0)) \leq \varepsilon^2.$$
\end{lemma}

\begin{proof}
Since $\zeta(\llbracket -\ell_0,+\ell_0 \rrbracket) \neq (0,\ldots,0)$, there exists $y \in \llbracket -\ell_0,+\ell_0 \rrbracket$ such that $\zeta(y)=1$. Choose arbitrarily such a $y$, and let $z$ be the minimum integer in $\llbracket y+2,+\infty \llbracket$ such that $\zeta(\llbracket z-1,z \rrbracket)=(0,0)$. By definition, we have that $z \leq x$, and, due to the definition of $z$ as a minimum and to the fact that $\zeta(y)=1$, one must have $\zeta(z-2)=1$, so that $\zeta(\llbracket z-2,z \rrbracket)=(1,0,0)$.  Using \eqref{d:epsilon}, we deduce that $\P_{\xi,\zeta}(Y_1(z-1)=0) \leq \varepsilon$. Defining $z'$ as the maximum integer in $\rrbracket -\infty,  y-2 \rrbracket$ such that $\zeta(\llbracket z',z'+1 \rrbracket)=(0,0)$, we deduce in a similar way that $z' \geq x'$, and that $\zeta(\llbracket z',z'+2 \rrbracket)=(0,0,1)$ so that $\P_{\xi,\zeta}(Y_1(z'+1)=0) \leq \varepsilon$. Since by construction $-2 \ell_0 \leq x'+1 \leq z'+1<z-1 \leq x-1 \leq 2 \ell_0$, we get that $\P_{\xi,\zeta}(Y_1(\llbracket -2\ell_0,+2 \ell_0 \rrbracket)=(0,\ldots,0)) \leq \P_{\xi,\zeta}(Y_1(z-1)=0) \P_{\xi,\zeta}(Y_1(z'+1)=0) \leq \varepsilon^2$.
\end{proof}

\begin{lemma}\label{l:non-determinisme}
Consider $\zeta_0,\zeta_1,\xi \in \{0,1\}^{\Z}$, $u<v$ and $1 \leq m \leq u-v+1=\ell$. If there is no subinterval of $\llbracket u,v\rrbracket$ of size $m$ on which $\phi[\zeta_0]$ and $\zeta_1$ coincide, then 
$$\P_{\xi,\zeta_0} \left(Y_1(\llbracket u,v\rrbracket) = \zeta_1(\llbracket u,v\rrbracket) \right) \leq  \left(2(1-p)\right)^{\pentinf{\ell/m}}.$$
\end{lemma}

\begin{proof}
Remember that 
\begin{equation}\label{e:produit-proba}\P_{\xi,\zeta_0} \left(Y_1(\llbracket u,v\rrbracket) =  \zeta_1(\llbracket u,v\rrbracket)\right) = \prod_{x \in \llbracket u,v\rrbracket} \P_{\xi,\zeta_0}(Y_1(x) = \zeta_1(x)).\end{equation}
Now consider $x$ such that $\zeta_1(x) \neq \phi[\zeta_0](x)$. By property \eqref{e:perturb-determ}, we have that $\P_{\xi,\zeta_0}(Y_1(x) = \zeta_1(x)) \leq 2 (1-p)$. Since the interval $\llbracket u,v\rrbracket$ contains $\pentinf{\ell/m}$ disjoint consecutive intervals of size $\ell$, every such subinterval on which $\phi[\zeta]$ and $\zeta_1$ do not coincide leads to at least one factor  $\P_{\xi,\zeta_0}(Y_1(x) = \zeta_1(x)) \leq 2 (1-p)$ in the product \eqref{e:produit-proba}, so we get the stated result.
\end{proof}

In the sequel, we use the notation $u_1= \left(2(1-p)\right)^{\pentinf{\ell_{0}/\ell_{1}}}$, and, for $n \geq 2$,  $u_n=  (4 \ell_0+1)  \left(2(1-p)\right)^{\pentinf{\ell_{n-1}/\ell_{n}}}$.

\begin{lemma}\label{l:antecedent-1}
Assume that $\zeta(\llbracket -\ell_0,+\ell_0 \rrbracket) \neq (0,\ldots,0)$, and that there is no subinterval of $\llbracket -2\ell_0,+2 \ell_0 \rrbracket$ of size $\ell_1$ on which $\zeta$ coincides with a sequence in {\normalfont $\overline{\d{1}} \cup \overline{\d{0}\d{1}}$}, where $\ell_1 \geq 1$. Then, for every $\xi$, $$\P_{\xi,\zeta}(Y_1(\llbracket -2\ell_0,+2 \ell_0 \rrbracket)=(0,\ldots,0)) \leq \max\left(\varepsilon^2 ,u_1\right).$$

\end{lemma}

\begin{proof}
Assume that there exists an $x \in \llbracket \ell_0+2, 2\ell_0+1 \rrbracket$ such that $\zeta(\llbracket x-1,x \rrbracket)=(0,0)$, and an $x' \in \llbracket -2\ell_0-1, -\ell_0-2 \rrbracket$ such that $\zeta(\llbracket x',x'+1 \rrbracket)=(0,0)$. By Lemma \ref{l:plage-de-zeros}, $\P_{\xi,\zeta}(Y_1(\llbracket -2\ell_0,+2 \ell_0 \rrbracket)=(0,\ldots,0)) \leq \varepsilon^2$. Now assume that there is no such $x'$, and that there is a sub-interval $\llbracket g,h \rrbracket $ of $\llbracket -2 \ell_0, -\ell_0 -1 \rrbracket$ of size $\ell_1$ on which  $\phi[\zeta]$ equals $(0,\ldots,0)$.  From the results of Table \ref{tab:antecedents}, $\zeta$ must coincide on $\llbracket g-1,h+1 \rrbracket $ with a sequence in $\overline{\d{0}} \cup \overline{\d{1}} \cup \overline{\d{0}\d{1}}$.  From the non-existence of $x'$, $\overline{\d{0}}$ is ruled out, and, given our assumption on $\zeta$, $\overline{\d{1}}$ and $\overline{\d{0}\d{1}}$ are ruled out too. Thus, there cannot be a  sub-interval of size $\ell_1$ of $\llbracket -2 \ell_0, -\ell_0 -1 \rrbracket$ on which  $\phi[\zeta]$ equals $(0,\ldots,0)$. By Lemma \ref{l:non-determinisme}, $\P_{\xi,\zeta}(Y_1(\llbracket -2\ell_0,-\ell_0-1 \rrbracket)=(0,\ldots,0)) \leq \left(2(1-p)\right)^{\pentinf{\ell_0/\ell_1}}$, and, since $\{Y_1(\llbracket -2\ell_0,+2 \ell_0 \rrbracket)=(0,\ldots,0))\} \subset  \{ Y_1(\llbracket -2\ell_0,-\ell_0-1 \rrbracket)=(0,\ldots,0)\}$, we deduce the upper bound $\P_{\xi,\zeta}(Y_1(\llbracket -2\ell_0,+2 \ell_0 \rrbracket)=(0,\ldots,0)) \leq \left(2(1-p)\right)^{\pentinf{\ell_0/\ell_1}}$.
A similar argument holds if we assume the non-existence of an $x  \in \llbracket \ell_0+2, 2\ell_0+1 \rrbracket$ such that $\zeta(\llbracket x-1,x \rrbracket)=(0,0)$.
\end{proof}

For $t \geq 1$, let $B_{t}$ denote the event that $Y_t(\llbracket -2\ell_0,+2 \ell_0 \rrbracket)=(0,\ldots,0)$ and  $Y_{t-1}(\llbracket -\ell_0,+\ell_0 \rrbracket) \neq (0,\ldots,0)$, and, for $1 \leq n \leq t$ and $\sigma \in \{0,1\}^{\Z}$, let 
$C_{\sigma,n,t}$ denote the event that there exists a sub-interval of $\llbracket -2\ell_0,+2 \ell_0 \rrbracket$ of size $\ell_n$ on which $Y_{t-n}$ coincides with $\sigma$, and let 
$B_{\sigma,n,t}$ denote the event that there exists a sub-interval of $\llbracket -2\ell_0,+2 \ell_0 \rrbracket$ of size $\ell_n$ on which $Y_{t-n}$ coincides with $\sigma$ and $Y_{t-n+1}$ coincides with $\phi[\sigma]$.  

From Lemma \ref{l:antecedent-1} and the union bound, using the Markov property at time $t-1$, we deduce that, for all $t \geq 1$, 
\begin{equation}\label{e:borne-1}
\P_{\xi,\zeta}(B_t) \leq \max\left(\varepsilon^2 ,u_1\right) +
\sum_{\sigma \in  \overline{\d{1}}  \cup  \overline{\d{0}\d{1}} }\P_{\xi,\zeta}(B_{\sigma,1,t}).
\end{equation}
Iterating the argument leading to \eqref{e:borne-1}, we obtain the following general result.
\begin{lemma}\label{l:antecedent-2}
For all $1 \leq n \leq t-1$, and $s \in \mathcal{P}'$, as soon as $\ell_{n+1} \geq 1$, 
 \begin{equation}\label{e:borne-2}\P_{\xi,\zeta}(B_{s,n,t}) \leq \P_{\xi,\zeta}(C_{s,n,t}) \leq \left(  \sum_{\sigma \in \phi^{-1}[s]} \P_{\xi,\zeta}(B_{\sigma,n+1,t}) \right) + u_{n+1}.\end{equation}
\end{lemma}

\begin{proof}
First note that $B_{s,n,t}$ is contained in the event $C_{s,n,t}$, whence the inequality $\P_{\xi,\zeta}(B_{s,n,t}) \leq \P_{\xi,\zeta}(C_{s,n,t})$.
Now decompose $C_{s,n,t}$ into the disjoint union of events $C_{s,n,t,w}$, where $w$ is the leftpoint of the leftmost subinterval $\llbracket u,v\rrbracket$ of $\llbracket -2\ell_0,+2 \ell_0 \rrbracket$ of size $\ell_n$ such that $Y_{t-n}(\llbracket u,v\rrbracket)=s(\llbracket u,v\rrbracket)$.
For such a $w$, let $w'=w+\ell_n-1$ and let $\mathcal{A}$ be the set of sequences $\eta \in \{0,1\}^{\Z}$ such that there is no subinterval of $\llbracket w,w' \rrbracket$ of size $\ell_{n+1}$ on which $\phi[\eta]$ and $s$ coincide.
Let $D_{s,n,t,w}$ denote the event that $Y_{t-n-1} \in \mathcal{A}$.  
By Lemma \ref{l:non-determinisme}, one has that, for every $\eta \in \mathcal{A}$, and every $\chi$,  $\P_{\chi,\eta}(Y_1(\llbracket w,w'\rrbracket)=s(\llbracket w,w'\rrbracket)) \leq  \left(2(1-p)\right)^{\pentinf{\ell_{n}/\ell_{n+1}}}$. 
As a consequence, using the Markov property at time $t-n-1$, we have the bound $\P_{\xi,\zeta}(C_{s,n,t,w} \cap D_{s,n,t,w}) \leq  \left(2(1-p)\right)^{\pentinf{\ell_{n}/\ell_{n+1}}}$. 
On the other hand, if $\eta \notin \mathcal{A}$,  there exists a subinterval $\llbracket g,h\rrbracket \subset \llbracket w,w'\rrbracket$ of size $\ell_{n+1}$ such that $\phi[\eta]$ and $s$ coincide on $\llbracket g,h\rrbracket$, and so there must exist 
$\sigma \in  \phi^{-1}[s]$ such that $\eta$ coincides with $\sigma$ on $\llbracket g-1,h-1\rrbracket$. As a consequence, 
\begin{equation}\label{e:grosse-inclusion}C_{s,n,t,w} \cap D_{s,n,t,w}^c \subset \bigcup_{\sigma} \bigcup_{\llbracket g,h\rrbracket}  \{Y_{t-n}(\llbracket g,h\rrbracket)=s(\llbracket g,h\rrbracket) \} \cap  \{ Y_{t-n-1}(\llbracket g,h\rrbracket)=\sigma(\llbracket g,h\rrbracket)  \},\end{equation}
where the union runs over all possible $\sigma \in  \phi^{-1}[s]$ and all possible 
$\llbracket g,h\rrbracket \subset \llbracket w,w'\rrbracket$ of size $\ell_{n+1}$. Now \eqref{e:grosse-inclusion} shows that 
$C_{s,n,t,w} \cap D_{s,n,t,w}^c \subset \bigcup_{\sigma \in \phi^{-1}[s]} B_{\sigma,n+1,t}$.
Since the $(C_{s,n,t,w})_w$ form a pairwise disjoint family, we deduce that 
$$\sum_{w} \P_{\xi,\zeta}\left(C_{s,n,t,w} \cap D_{s,n,t,w}^c\right) \leq \P_{\xi,\zeta} \left(    \bigcup_{\sigma \in \phi^{-1}[s]} B_{\sigma,n+1,t} \right),$$
whence, using the union bound:
\begin{equation}\label{e:gros-bout-1}\sum_{w} \P_{\xi,\zeta}\left(C_{s,n,t,w} \cap D_{s,n,t,w}^c\right) \leq \sum_{\sigma \in \phi^{-1}[s]} \P_{\xi,\zeta}(B_{\sigma,n+1,t})\end{equation}
On the other hand, we have seen that $\P_{\xi,\zeta}(C_{s,n,t,w} \cap D_{s,n,t,w}) \leq  \left(2(1-p)\right)^{\pentinf{\ell_{n}/\ell_{n+1}}}$, so that, taking into account the fact that the number of possible $w$ is bounded above by $ (4 \ell_0+1)$, 
\begin{equation}\label{e:gros-bout-2}\sum_{w} \P_{\xi,\zeta}\left(C_{s,n,t,w} \cap D_{s,n,t,w}\right) \leq  (4 \ell_0+1) \left(2(1-p)\right)^{\pentinf{\ell_{n}/\ell_{n+1}}}.\end{equation}
To conclude, remember that $\P_{\xi,\zeta}(B_{s,n,t}) \leq \P_{\xi,\zeta}(C_{s,n,t}) = \sum_{w} \P_{\xi,\zeta}(C_{s,n,t,w})$, and combine \eqref{e:gros-bout-1} and  \eqref{e:gros-bout-2}.
\end{proof}
Using Lemma \ref{l:antecedent-2} to bound  $\P(B_{\sigma,1,t})$ for $\sigma \in \overline{\d{1}} \cup \overline{\d{0}\d{1}}$, \eqref{e:borne-1} leads to the following bound, valid for all $t \geq 2$,
\begin{equation}\label{e:biborne-3}
\P_{\xi,\zeta}(B_t) \leq \max\left(\varepsilon^2 , u_1 \right) + 3  u_2 + 
\sum_{\sigma \in \mbox{\normalfont \scriptsize $\overline{\d{1}\d{1}\d{0}\d{0}} \cup \overline{\d{1}\d{0}\d{0}\d{0}} \cup \overline{\d{0}\d{1}\d{1}\d{1}}$}}\P_{\xi,\zeta}(B_{\sigma,2,t})
\end{equation}
Using again Lemma \ref{l:antecedent-2} to bound $\P_{\xi,\zeta}(B_{\sigma,2,t})$ for $\sigma \in \overline{\d{0}\d{1}\d{1}\d{1}}$, \eqref{e:biborne-3} leads to the following bound, valid for all $t \geq 3$,
\begin{equation}\label{e:borne-3}
\P_{\xi,\zeta}(B_t) \leq \max\left(\varepsilon^2 , u_1\right) + 3 u_2 + 4 u_3 +
\sum_{\sigma \in \mbox{\normalfont \scriptsize $\overline{\d{1}\d{1}\d{0}\d{0}} \cup \overline{\d{1}\d{0}\d{0}\d{0}}$} } \P_{\xi,\zeta}(B_{\sigma,2,t})+ \sum_{\sigma \in \mbox{\normalfont \scriptsize $\overline{\d{0}\d{0}\d{0}\d{1}\d{1}\d{0}\d{1}\d{1}} \cup \overline{\d{1}\d{1}\d{1}\d{0}\d{0}\d{1}\d{0}\d{0}}$} }\P_{\xi,\zeta}(B_{\sigma,3,t}).
\end{equation}
Observe that \eqref{e:borne-3} involves sums that run over elements of $\mathcal{P''}$ only, where $$\mathcal{P''}=\overline{\d{1}\d{1}\d{0}\d{0}} \cup \overline{\d{1}\d{0}\d{0}\d{0}} \cup \overline{\d{0}\d{0}\d{0}\d{1}\d{1}\d{0}\d{1}\d{1}} \cup \overline{\d{1}\d{1}\d{1}\d{0}\d{0}\d{1}\d{0}\d{0}}.$$
It turns out that elements of $\mathcal{P''}$ combine nicely with our assumption on the dynamics. Let $H_{n,t}$ denote the event that $X_{t-n}(\llbracket -2 \ell_0,+2 \ell_0  \rrbracket)$ is $\ell_n-$nice. In the sequel, we use the notation $v_n=(4 \ell_0+1) (1-\varepsilon)^{\frac{1}{9}\pentinf{(\ell_n-2)/8} }$. 

\begin{lemma}\label{l:antecedent-3}
For all $1 \leq n \leq t$, and $s \in \mathcal{P}''$, as soon as $\ell_n \geq 10$, 
 \begin{equation}\label{e:borne-4}\P_{\xi,\zeta}(B_{s,n,t} \cap H_{n,t} ) \leq v_n.\end{equation}
\end{lemma}

\begin{proof}
Let $\gamma_1=(0,0,1)$ and $\gamma_2=(1,0,0)$, and let $r^* \in \{1,2 \}$ be such that $q((a^*,b^*,c^*),\gamma_{r^*})=\varepsilon$, and consider a  subinterval $\llbracket g,h\rrbracket$ of 
 $\llbracket -2\ell_0,+2 \ell_0 \rrbracket$ of size $\ell_n$.  From the definition of $\mathcal{P}''$, we see that there exists an 
$i \in \{1,\ldots, 8\}$ such that, for all $x \in  \llbracket g, h-2 \rrbracket \cap (i+8 \Z)$,  $s(x,x+1,x+2)=\gamma_{r^*}$, and $\phi[s](x+1)=1$. 
Using the definition of $r^*$, we deduce that, for all $x \in  \llbracket g, h-2 \rrbracket \cap (i+8 \Z)$, on the intersection of the two events $X_{t-n}(x,x+1,x+2) =(a^*,b^*,c^*)$ and $Y_{t-n}(x,x+1,x+2)=s(x,x+1,x+2)$, we have that 
 $\P_{\xi,\zeta}(  Y_{t-n+1}(x+1)=\phi[s](x+1) | X_{t-n},Y_{t-n} ) = (1-\varepsilon)$. Now, on  $H_{n,t}$, we have that at least one ninth of the $x \in  \llbracket g, h-2 \rrbracket \cap (i+8 \Z)$ are such that $X_{t-n}(x,x+1,x+2) =(a^*,b^*,c^*)$. Since there are at least $\pentinf{(\ell_n-2)/8}$ integers in $\llbracket g, h-2 \rrbracket \cap (i+8 \Z)$, we deduce that, on the intersection of the two events $Y_{t-n}(\llbracket g,h \rrbracket)=s(\llbracket g,h \rrbracket)$ and $H_{n,t}$, we have the bound
 $\P_{\xi,\zeta}(  Y_{t-n+1}(\llbracket g, h \rrbracket) = \phi[s](\llbracket g, h \rrbracket)  | X_{t-n},Y_{t-n} ) \leq (1-\varepsilon)^{ \frac{1}{9}\pentinf{(\ell_n-2)/8}  }$, so that 
 $$\P_{\xi,\zeta} \left(\{ Y_{t-n+1}(\llbracket g, h \rrbracket) = \phi[s](\llbracket g, h \rrbracket) \} \cap \{  Y_{t-n}(\llbracket g, h \rrbracket) = s(\llbracket g, h \rrbracket) \} \cap H_{n,t} \right) \leq (1-\varepsilon)^{ \frac{1}{9}\pentinf{(\ell_n-2)/8}  }.$$
 Taking the union bound over the $(4 \ell_0+1)$ possible intervals $\llbracket g,  h\rrbracket$, we get the desired bound.
 \end{proof}
Combining \eqref{e:borne-3} and \eqref{e:borne-4} we deduce that, for $t \geq 3$,  
\begin{equation}\label{e:borne-5}
\P_{\xi,\zeta}(B_t) \leq \max\left(\varepsilon^2 , u_1\right) + 3 u_2+ 4 u_3+
 8 \left(\P_{\xi,\zeta}(H_{2,t}^c) + v_2 \right) + 16  \left(\P_{\xi,\zeta}(H_{3,t}^c) + v_3 \right).
\end{equation}
Arguing in the same way as for $B_t$, we can combine Lemma \ref{l:antecedent-2} and Lemma \ref{l:antecedent-3} to obtain the following bounds. 
For $s =  \overline{\d{1}}$, 
 \begin{equation}\label{e:biborne-4a}
\P_{\xi,\zeta}(C_{s,1,t}) \leq  u_2 +
 4 \left(\P_{\xi,\zeta}(H_{2,t}^c) + v_2\right). 
\end{equation}
For $s \in \overline{\d{0}\d{1}}$, 
\begin{equation}\label{e:biborne-4b}
\P_{\xi,\zeta}(C_{s,1,t}) \leq  u_2 +2 u_3 + 2 \left(\P_{\xi,\zeta}(H_{2,t}^c) +v_2 \right)+8 \left(\P_{\xi,\zeta}(H_{3,t}^c) + v_3\right). 
\end{equation}
For $s \in \overline{\d{0}\d{1}\d{1}\d{1}}$, 
\begin{equation}\label{e:biborne-4c}
\P_{\xi,\zeta}(C_{s,2,t}) \leq  u_3 + 4 \left(\P_{\xi,\zeta}(H_{3,t}^c) + v_3 \right). 
\end{equation}

\subsection{Dynamics of the rightmost $1$}

We start with an initial configuration $Y_0=\zeta$ satisfying $\zeta(0)=1$, $\zeta(x)=0$ for $x \geq 1$, and such that the total number of $x$ such that $\zeta(x)=1$ is finite. Due to \eqref{e:zero-zero}, we have that, almost surely, the number of $x$ such that $Y_n(x)=1$ is finite for all $n \geq 1$. We then let
$$R_n = \sup \{ x \in \Z; \ Y_n(\llbracket x,x+2 \rrbracket)  \} = (1,0,0) ,  \  L_n = \sup \{ x \leq R_n ; \ Y_n(\llbracket x-2,x \rrbracket)  \} = (0,0,1)\},$$
with the usual convention that $\sup \emptyset = -\infty$. Note that $R_n$ coincides with the location of the righmost $1$ in $Y_n$, while $L_n$ coincides with the location of the $1$ closest to $R_n$ that admits two consecutive $0$s at its left. With our assumptions, one has that $-\infty<L_0 \leq R_0=0$, and, moreover, provided that $Y_n  \neq \overline{0}$, using the fact that  there is a.s. a finite number of $x$ such that $Y_n(x)=1$,
we have that a.s. $-\infty < L_n \leq R_n < +\infty$. We then let $$G_{n+1} =\sup \{  m \geq 1 ;   Y_{n+1}(\llbracket   R_n-(m-1) , R_n \rrbracket)     = (0,\ldots, 0)  \},$$ with the convention that $\sup \emptyset = 0$ (for the sake of definiteness, we also let $G_{n+1}=0$ when $R_n=+\infty$).
On the event $\{Y_n  \neq \overline{0}, \ R_n<+\infty, \ R_{n+1}=R_n+1\}$, we have that $G_{n+1}$ is the size of the gap (interval of $0$s) between the rightmost $1$ and the second rightmost $1$ in $Y_{n+1}$ ($+\infty$ if there is no second rightmost $1$).
On the event $\{Y_n  \neq \overline{0}, \ R_{n}<+\infty, \ R_{n+1} \neq R_n+1 \}$, we have that a.s. $R_{n+1}=R_n-G_{n+1}$. 

We now define the event $C_{n+1}$ by
$$C_{n+1}= \{ R_{n+1} = R_n+1 \} \cup \{Y_{n+1}(L_n-1) = 1  \},$$ where, for the sake of definiteness, we let $Y_{n+1}(\pm \infty) =0$. 
From \eqref{d:epsilon} and the definition of $R_n$ and $L_n$, we deduce the two following bounds:
\begin{equation}\label{e:key-1} \mbox{On $\{-\infty < L_n \leq R_n < +\infty\}$, one has that }  \P_{\xi,\zeta}(C_{n+1}^c|X_n,Y_n) \leq \varepsilon^2.\end{equation}
\begin{equation}\label{e:key-2} \mbox{On $\{-\infty < L_n \leq R_n < +\infty\}$, one has that }  \P_{\xi,\zeta}(R_{n+1} \neq R_n+1 |X_n,Y_n,G_{n+1}) \leq \varepsilon.\end{equation}

\begin{lemma}\label{l:danger}
For all $n \geq 0$, one has the following inclusion between events as soon as $\lambda \geq 1$ (see Section \ref{s:param-list}):
$$\{G_{n+1} \geq 2 \lambda , R_{n+1}=R_n+1\} \cap  \bigcap_{k=1}^{\lambda} C_{n+1+k}  \subset \bigcap_{k=1}^{\lambda} \{G_{n+1+k} \leq 2 \lambda - 1 \} \cup \{R_{n+1+k}=R_{n+k}+1 \}.$$ 
\end{lemma}

\begin{proof}
Assuming that $G_{n+1} \geq 2 \lambda$ and $R_{n+1}=R_n+1$, we also have that $L_{n+1}=R_n+1$, and the configuration of $1$s in $Y_{n+1+k}(\llbracket R_n- 2\lambda+k+1, +\infty \llbracket)$ for $k=1,\ldots, \lambda$ is formed by the "offspring" of the single $1$ at $R_{n}+1$ at time $n+1$. As a consequence of $C_{n+2} \cap \cdots \cap C_{n+\lambda}$, this "offspring" is present at times $t=n+2,\ldots, n+\lambda$, so we have, for all $k=1,\ldots, \lambda$, that $R_{n}+1-(k-1) \leq L_{n+k} \leq R_{n+k} \leq R_{n}+1+(k-1)$. Moreover, whenever $R_{n+k+1} \neq R_{n+k}+1$, the occurrence of $C_{n+k+1}$ implies that we have
$Y_{n+k+1}(L_{n+k}-1)=1$, so that $G_{n+k+1} \leq R_{n+k}-L_{n+k}+1 \leq 2 \lambda-1 $.
\end{proof}
As in \cite{BraNeu} (Section 3), we use \eqref{e:key-1}, \eqref{e:key-2} and Lemma \ref{l:danger} to couple $R_0,\ldots, R_T$ with a random walk $V_0,\ldots, V_T$ on $\Z$, started at $V_0=0$, with i.i.d. steps whose distribution is as follows:   $$\left\lbrace \begin{array}{ll} +1 & \mbox{ with probability }1-\varepsilon \\ -2 \lambda & \mbox{ with probability } \varepsilon \end{array}\right. ,$$
and with the following estimate:
\begin{equation}\label{l:compare-prob}\P_{\xi,\zeta}\left(  \vphantom{\sum} \left\{  \forall i \in \llbracket 0,T \rrbracket,   \   R_i \geq V_i  \right\}^c  \right) \leq T \varepsilon^2+\pentsup{T/(\lambda+1)} \varepsilon.\end{equation}

We shall then need two estimates, corresponding\footnote{We could not find a proper reference for the large deviations bound $P(V_{4L}<12L/5) \leq C_1 e^{-\gamma_1 L}$ in \cite{BraNeu}, which is quoted there without proof. Here, we give a complete proof, with a bound leading to slightly worse (but sufficient for our purposes) estimates.} to Lemma 3.14 in \cite{BraNeu}. 

\begin{lemma}\label{l:grandes-dev}
As soon as $(1+2 \lambda) \varepsilon \leq 1/5$, 
\begin{equation}\label{e:grandes-dev-1}\P_{\xi,\zeta}(V_T \leq 12K/5) \leq e^{-3 \left(\frac{1/5}{1+2 \lambda}-\varepsilon/2 \right) K }.\end{equation}
As soon as $(1+2 \lambda) \varepsilon \leq 1/40$,
\begin{equation}\label{e:grandes-dev-2}\P_{\xi,\zeta}(V_n \leq -K/5 \mbox{ \normalfont  for some }1 \leq n \leq T) \leq  T e^{-3 \left(\frac{1/40}{1+2 \lambda}-\varepsilon/2 \right) K }.\end{equation}

\end{lemma}

\begin{proof}
For $0 \leq n \leq T$, let us denote by $N_n$ the number of indices $0 \leq k \leq n$ such that $V_k=-2 \lambda$. Thus, we can write $V_n=-(2\lambda)N_n+(n-N_n)=n-(1+2 \lambda)N_n$, so we have that, for every $\ell$, 
$$\P_{\xi,\zeta}(V_n \leq \ell) = \P_{\xi,\zeta} \left( N_n \geq \frac{n-\ell}{1+2 \lambda} \right).$$
Noting that $\E(V_n)=\varepsilon n$, and using Theorem 2.3 (b) of \cite{McD}, we have that, as soon as $u \geq 1$, 
\begin{equation}\label{e:mcd}\P_{\xi,\zeta}( V_n \geq (1+u) \varepsilon n   ) \leq  e^{-\frac{3}{8} u n \varepsilon}.\end{equation}
Using $n=T=4K$, $\ell=12K/5$, and $u$ defined by $(1+u) \varepsilon n = \frac{n-\ell}{1+2 \lambda}$,  we deduce \eqref{e:grandes-dev-1}, where we use the condition  $(1+2 \lambda) \varepsilon \leq 1/5$ to ensure that $u \geq 1$. To get \eqref{e:grandes-dev-2}, we note that, by the union bound, $\P_{\xi,\zeta}(V_n \leq -K/5 \mbox{ for some }1 \leq n \leq T) \leq \sum_{n=1}^{T} \P_{\xi,\zeta}(V_n \leq -K/5)$, and observe that, for every $1 \leq n \leq T$, the bound provided by \eqref{e:mcd}, with $\ell=-K/5$ and $u$ defined by $(1+u) \varepsilon n =  \frac{n-\ell}{1+2 \lambda}$, is bounded above by $e^{-3 \left(\frac{1/40}{1+2 \lambda}-\varepsilon/2 \right) K }$, with the condition  $(1+2 \lambda) \varepsilon \leq 1/40$ ensuring that $u \geq 1$.
\end{proof}

\section{Proof of Theorem \ref{t:letheo}}\label{s:preuve-theo}

In this section, we put together the different pieces leading to the proof of Theorem \ref{t:letheo}. Following \cite{BraNeu}, the key idea is to show that, working on a sufficiently large scale, the process of occurrence of $1$s in $Y_t$ can be compared to super-critical oriented percolation. To this end, we introduce the following definition.

\begin{madef}
A pair $(z,k)$ with $z \in \Z$ and $k \in \N$ is said to be good if:
 \begin{itemize}
 \item[(i)] $Y_{kT}(zK+I) \neq (0,\ldots, 0),$
 \item[(ii)] $Y_{kT}(zK+J_{-}) \notin \mathcal{G} , Y_{kT}(zK+J_{+}) \notin \mathcal{G}$,
  \item[(iii)] $Q_{\ell_3} \left( (X_{kT}(zK+x))_{x \in \Z} \right) \leq \theta$ ,
 \end{itemize}
 where $\mathcal{G}$ is defined as the set of $\{0,1\}-$valued sequences which coincide with an element of $\overline{\d{\normalfont 1}} \cup \overline{\d{\normalfont 0}\d{\normalfont 1}}$ on an interval of size $\geq \ell_1$, or which coincide with an element of $\overline{\d{\normalfont 0}\d{\normalfont 1}\d{\normalfont 1}\d{\normalfont1}}$ on an interval of size $\geq \ell_2$. 
\end{madef}

We first note that there exist initial conditions $\xi,\zeta$ such that $(0,0)$ is good. This is obvious for conditions (i) and (ii), taking e.g. $\zeta(x)$ equal to $1$ for $x=0$, and $\zeta(x)=0$ for $x \neq 0$. As regards (iii), starting with any initial condition $\xi$, inequality \eqref{e:merci-Markov} shows that $\P_{\xi} \left( Q_{\ell_3}(X_K) > \theta  \right) \leq  \theta^{-1} K (e_{K,\ell_3} + |J|d_K)$, so that, since the r.h.s. goes to zero as $p \to 1
$, uniformly with respect to $F$, suitable initial conditions exist as soon as $p$ is close enough to $1$. In the sequel, we implicitly assume that all the configurations $\zeta$ appearing in the various results have a finite number of $1$s (provided that $Y_0$ has this property, it is a.s. the case for every subsequent $Y_t$).

We now have to prove the key result enabling the comparison with oriented percolation. 
\begin{prop}\label{p:good}
If $(0,0)$ is good, $\lim_{p \to 1}\P_{\xi,\zeta} \left(\mbox{\normalfont $(1,1)$ is good} \right)=1$, where the limit is uniform w.r.t. $\xi,\zeta,F$.
\end{prop}

To establish Proposition \ref{p:good}, we follow the two-step strategy of \cite{BraNeu}, first showing that the interval $Y_t(J_+^2)$ has a large probability of containing a $1$ at a certain time $0 \leq t \leq T$ (Section 3 of \cite{BraNeu}), then using a "repositioning algorithm" to deduce that $Y_T(\llbracket K,2K\rrbracket)$ has a large probability of containing a $1$ (Section 4 of \cite{BraNeu}). The proofs follow \cite{BraNeu} rather closely, so we allow ourselves to be sketchy here and there, referring to \cite{BraNeu} for additional explanations.

\begin{lemma}\label{l:peuplement-1}
If $(0,0)$ is good, $\lim_{p \to 1}\P_{\xi,\zeta}\left(\mbox{\normalfont For some }0 \leq t \leq T, Y_t(J^2_+) \neq (0,\ldots,0)\right)=1$, unif. w.r.t. $\xi,\zeta,F$.
\end{lemma}

\begin{proof}
If $Y_0(J^2_+) \neq (0,\ldots,0)$ we are done, so let us assume that $Y_0(J^2_+) = (0,\ldots,0)$. Since $(0,0)$ is good (property (i)), we have that $Y_0(I) \neq (0,\ldots, 0)$, and the rightmost $1$ left of $J^2_+$ in $Y_0$ is located at a site $r \in \Z$ satisfying $r \geq -K$. Until $1$s enter into $J^2_+$, the dynamics of $Y_t$ right and left of $J^2_+$ are independent, so we only have to establish the conclusion of the Lemma in the case where $r$ is the position of the rightmost $1$ in $Y_0$. In such a case, we can use  \eqref{l:compare-prob} and the corresponding large deviations estimate \eqref{e:grandes-dev-1} from Lemma \ref{l:grandes-dev} to establish the conclusion, since $r+12K/5 \geq -K+12K/5=7K/5$, which is the left border of $J^2_+$. To be specific, we obtain the following bound, valid as soon as $(1+2 \lambda) \varepsilon \leq 1/5$:
\begin{equation}\label{e:grosse-borne}\P_{\xi,\zeta}\left(\mbox{For some }0 \leq t \leq T, Y_t(J^2_+) \neq (0,\ldots,0)\right) \leq T \varepsilon^2+\pentsup{T/(\lambda+1)} \varepsilon +  e^{-3 \left(\frac{1/5}{1+2 \lambda}-\varepsilon/2 \right) K }.\end{equation}
 Using the explicit expressions of $K$, $T$, $\lambda$ as functions of $\varepsilon$, and the fact that $\varepsilon \leq 2(1-p)$, it is now apparent that the r.h.s. of \eqref{e:grosse-borne} goes to $0$ as $p \to 1$, uniformly with respect to $\xi,\zeta,F$.
\end{proof}

\begin{lemma}\label{l:peuplement-2}
If $(0,0)$ is good, $\lim_{p \to 1}\P_{\xi,\zeta}\left(Y_T(\llbracket K, 2K \rrbracket) \neq (0,\ldots,0)\right)=1$, uniformly w.r.t. $\xi,\zeta,F$.
\end{lemma}

\begin{proof}
Here is a sketch of the argument. By Lemma \ref{l:peuplement-1}, with high probability, there is a smallest $t \in \llbracket 0,T \rrbracket$ such that $Y_t(J^2_+) \neq (0,\ldots,0)$, that we denote by $\tau$.  Then let $\sigma_1$ be the first $t \in \llbracket \tau+1,T \rrbracket$ such that  $Y_t(J^2_+) = (0,\ldots,0)$ (if there exists such a $t$). Then the dynamics of $Y_t$ left and right of $3K/2$ are independent until the first time posterior to $\sigma_1$ where $Y_t(3K/2)=1$, denoted by $\tau_1$; moreover, up to time $\tau_1$, these dynamics coincide with those comprising no $1$ right (resp. left) of $3K/2$ at time $\sigma_1$. If $Y_{\sigma_1}(J^1_+ \cup J^3_+) \neq 0$, we can use \eqref{l:compare-prob} and the corresponding large deviations estimates to show that, with high probability, $Y_t(\llbracket K,2K \rrbracket) \neq (0,\ldots,0)$ for $t=\sigma_1+1,\ldots, \tau_1$, while (deterministically) $\tau_1 \geq \sigma_1+K/10$ since $1$s move by nearest-neighbour steps. 
One then defines $\sigma_2$ as the first time posterior to $\tau_1$ such that  $Y_t(J^2_+) = (0,\ldots,0)$, and then define $\tau_2,\sigma_3,\tau_3,\sigma_4,\ldots$ accordingly, iterating the argument. Thus, the two types of unfavorable events that can make the repositioning algorithm fail are:  
\begin{itemize}
\item[\textbullet] failure to have $Y_{\sigma_i}(J^1_+ \cup J^3_+) \neq 0$ while $Y_{\sigma_i}(J^2_+) = (0,\ldots,0)$;
\item[\textbullet] failure of the righmost $1$ left of $J^2_+$ in $Y_{\sigma_i}$ (or, if there is none, of the leftmost $1$ right of $J^2_+$) to remain within $\llbracket K,2K\rrbracket$ until there is again a $1$ at $3K/2$ (or until $T$, if this does not happen before time $T$), we denote this event by $A_i$.
\end{itemize}    
Moreover, there can be at most $\frac{T}{(K/10)}=40$ indices $i$ such that $\sigma_i \leq T$.
Thus, we obtain the following bound: 
\begin{align}\label{e:reposition}  \P_{\xi,\zeta}\left(Y_T(\llbracket K, 2K \rrbracket) \neq (0,\ldots,0)\right) & \leq \\ \nonumber & \P_{\xi,\zeta}\left(\mbox{For all }0 \leq t \leq T, Y_t(J^2_+) = (0,\ldots,0)\right) \\ \nonumber
+ & \P_{\xi,\zeta}(\mbox{For some } 1 \leq i \leq 40, \ A_i \mbox{ occurs} ) \\ \nonumber + & \P_{\xi,\zeta}\left(\mbox{For some } 1 \leq t \leq T  ,  \   Y_{t-1}(J^1_+ \cup J^3_+) \neq 0 \mbox{ and }Y_{t}(J^2_+) = (0,\ldots,0) \right).
\end{align}
By the union bound and  \eqref{l:compare-prob}, 
$$ \P_{\xi,\zeta}(\mbox{For some } 1 \leq i \leq 40, \ A_i \mbox{ occurs} ) \leq 40 \left( T \varepsilon^2+\pentsup{T/(\lambda+1)} \varepsilon  +  \P_{\xi,\zeta}(V_n \leq -K/5 \mbox{ for some }1 \leq n \leq T) \right),$$
and, thanks to \eqref{e:grandes-dev-2}, we deduce that, as soon as $(1+2 \lambda) \varepsilon \leq 1/40$, 
\begin{equation}\label{e:un-bout}  \P_{\xi,\zeta}(\mbox{For some } 1 \leq i \leq 40, \ A_i \mbox{ occurs} ) \leq 40 \left( T \varepsilon^2+\pentsup{T/(\lambda+1)} \varepsilon +  T e^{-3 \left(\frac{1/40}{1+2 \lambda}-\varepsilon/2 \right) K } \right).\end{equation}

On the other hand, using the assumption that $(0,0)$ is good (property (ii)) together with Lemma \ref{l:antecedent-1}, we have that 
\begin{equation}\label{e:borne-6}\P_{\xi,\zeta}\left( Y_{0}(J^1_+ \cup J^3_+) \neq 0 \mbox{ and }Y_{1}(J^2_+) = (0,\ldots,0) \right) \leq\max\left(\varepsilon^2 , u_1\right) =: \Psi_1,\end{equation}
Now using \eqref{e:controle-nice} and the assumption that $(0,0)$ is good (property (iii)), the probability that there exists a $0 \leq t \leq T-1$ such that $X_t(J)$ is not $\ell_3-$nice, is bounded above by  
$\theta + 3 K (e_{K,\ell_3} + |J|d_K)$. Moreover, $\ell_3-$nice implies $\ell_2-$nice since $\ell_3 \leq \ell_2$. In particular, as soon as $\theta + 3 K (e_{K,\ell_3} + |J|d_K)<1$, the assumption that $(0,0)$ is good implies that $X_0(J)$ is $\ell_2-$nice. Using again the assumption that $(0,0)$ is good  (property (ii)) together with \eqref{e:biborne-3}, Lemma \ref{l:antecedent-3} and translation invariance, we deduce that 
\begin{equation}\label{e:borne-7}\P_{\xi,\zeta}\left( Y_{1}(J^1_+ \cup J^3_+) \neq 0 \mbox{ and }Y_{2}(J^2_+) = (0,\ldots,0) \right) \leq  \max\left(\varepsilon^2 , u_1\right) + 3 u_2 + 8 v_2 =: \Psi_2.\end{equation}
 Similarly, using \eqref{e:borne-3}, Lemma \ref{l:antecedent-3} and \eqref{e:controle-nice}, we deduce that, for $t=3,\ldots, T$, 
\begin{equation}\label{e:borne-8}\P_{\xi,\zeta}\left( Y_{t-1}(J^1_+ \cup J^3_+) \neq 0 \mbox{ and }Y_{t}(J^2_+) = (0,\ldots,0) \right) \leq \Psi_3,\end{equation}
where $\Psi_3 :=  \max\left(\varepsilon^2 , u_1 \right) + 3 u_2 + 4 u_3+
 8 v_2 + 16 v_3 + 24 (\theta + 3 K (e_{K,\ell} + |J|d_K))$. By the union bound, 
\begin{equation}\label{e:un-deuxieme-bout} \P_{\xi,\zeta}\left(\mbox{For some } 1 \leq t \leq T  ,  \   Y_{t-1}(J^1_+ \cup J^3_+) \neq 0 \mbox{ and }Y_{t}(J^2_+) = (0,\ldots,0) \right) \leq \Psi_1+\Psi_2+(T-2) \Psi_3.\end{equation}

Putting together \eqref{e:reposition}, the conclusion of Lemma \ref{l:peuplement-1}, \eqref{e:un-bout} and \eqref{e:un-deuxieme-bout}, we deduce that $\P_{\xi,\zeta}\left(Y_T(\llbracket K, 2K \rrbracket) \neq (0,\ldots,0)\right)$ goes to $0$ as $p \to 1$, uniformly with respect to $\xi,\zeta,F$.
\end{proof}

\begin{lemma}\label{l:peuplement-3}
If $(0,0)$ is good, $\lim_{p \to 1}\P_{\xi,\zeta}\left(Y_{T}(J_{+}) \notin \mathcal{G} \right)=1$, where the limit is uniform with respect to $\xi,\zeta,F$.
\end{lemma}

\begin{proof}
Using \eqref{e:controle-nice-2}, we have that  
$$\P_{\xi}\left( \bigcup_{t=T-2}^{T-1} \{X_t(J+2K)\mbox{ is not $\ell_3-$nice} \}\right) \leq 2( e_{K,\ell_3} + |J|d_K).$$
Using translation invariance and respectively \eqref{e:biborne-4a}, \eqref{e:biborne-4b}, \eqref{e:biborne-4c}, we deduce the following estimates:
$$\P_{\xi,\zeta}\left(Y_{T}(J_{+}) \mbox{ coincides with $\overline{\d{1}}$ on an interval of size $\geq \ell_1$}  \right)  \leq 
u_2+ 4v_2  +  8( e_{K,\ell_3} + |J|d_K ).$$
$$\P_{\xi,\zeta}\left(Y_{T}(J_{+}) \mbox{ coincides with $\overline{\d{0}\d{1}}$ on an interval of size $\geq \ell_1$}  \right)  \leq  2u_2 +4 u_3 + 4 v_2 + 16  v_3 +  40( e_{K,\ell_3} + |J|d_K ).$$
$$\P_{\xi,\zeta}\left(Y_{T}(J_{+}) \mbox{ coincides with $\overline{\d{0}\d{1}\d{1}\d{1}}$ on an interval of size $\geq \ell_2$}  \right) \leq  4u_3 + 16 v_3 + 32 ( e_{K,\ell_3} + |J|d_K ).$$
Each of the three upper bounds goes to $0$ as $p \to 1$, uniformly with respect to $\xi,\zeta,F$, which establishes the conclusion.
\end{proof}

All the ingredients needed to prove Proposition \ref{p:good} are now available.
\begin{proof}[Proof of Proposition \ref{p:good}]
Assuming that $(0,0)$ is good, thanks to Lemma \ref{l:peuplement-2}, we see that condition (i) is satisfied for $(1,1)$ with probability going to $1$ as $p \to 1$, uniformly with respect to $\xi,\zeta,F$.
Lemma \ref{l:peuplement-3} yields the same conclusion for condition (ii). As for condition (iii), the conclusion stems from \eqref{e:controle-nice-3}.
\end{proof}

We are now in a position to prove Theorem \ref{t:letheo}. Observe that, using a completely analogous argument as the one leading to Proposition \ref{p:good}, we obtain that, if that $(0,0)$ is good, the probability that $(-1,1)$ is good goes to $1$ as $p \to 1$, uniformly with respect to $\xi,\zeta,F$. (Here, $J_-$ plays the role of $J_+$, and we can define $J_-^1,J_-^2, J_-^3$ symmetrically to  $J_+^1,J_+^2, J_+^3$).
By the union bound, we deduce that, if $(0,0)$ is good, the probability that both $(+1,1)$ and $(-1,1)$ are good, goes to $1$ as $p \to 1$, uniformly with respect to $\xi,\zeta,F$.
Let us write this a bit more formally as follows: if $(0,0)$ is good, then
\begin{equation}\label{e:unif-zero} \P_{\xi,\zeta}\left( (+1,1) \mbox{ and }(-1,1) \mbox{ are good} \right) \geq 1 - \iota(p),\end{equation}
where $\iota$ is a function of $p$ only, satisfying $\lim_{p \to 1} \iota(p)=0$.
For $k \geq 0$, let $\F_{k}$ be the $\sigma-$algebra generated by the events of the form $\{ (z,n) \mbox{ is good} \}$ for $z \in \Z$ and $0 \leq n \leq k$. By the Markov property of the dynamics and translation invariance, we deduce from \eqref{e:unif-zero} that, for all $z \in \Z$ and $k \geq 0$ on the event $\{ (z,k) \mbox{ is good} \}$, 
\begin{equation}\label{e:unif-general} \P_{\xi,\zeta}\left((z+1,k+1) \mbox{ and }(z-1,k+1) \mbox{ are good}  | \F_{n} \right) \geq 1 - \iota(p).\end{equation}
Moreover, conditional on $\F_k$, the event $\{ (z-1,k+1) \mbox{ and }(z+1,k+1) \mbox{ are good} \}$ is independent\footnote{The condition that $(z,k+1)$ is good involves at most the values of $Y_{(k+1)T}(zK+\llbracket-K,K\rrbracket)$ (condition (i)), $Y_{(k+1)T}(zK+\llbracket-2K,2K\rrbracket)$ (condition (ii)), and $Y_{(k+1)T}(zK+\llbracket-3K,3K\rrbracket)$ (condition (iii)). So the event that both $(z-1,k+1)$ and $(z+1,k+1)$ are good involves at most the value of $Y_{(k+1)T}(zK+\llbracket-4K,4K\rrbracket)$. Conditional on $X_{kT},Y_{kT}$, the nearest-neighbour character of the dynamics shows that we have independence as soon as $|zK-z'K| \geq 2(4K+T)=16K$.} from the family of events $\{ (z'-1,k+1) \mbox{ and }(z'+1,k+1) \mbox{ are good} \}$, where $z' \in \Z \setminus \llbracket z-15,z+15 \rrbracket$.
As a consequence, combining \eqref{e:unif-general} and the classic result of \cite{LigSchSta} on domination by product measures, we deduce that, for $p$ sufficiently close to $1$, if $(0,0)$ is good, the random field $\left[\un( (z,k) \mbox{ is good} )\right]_{z \in \Z, k \geq 0}$ stochastically dominates a directed percolation process on $\Z^2$ in which, with probability $\kappa(p)$, $(x,k)$ is connected to both $(x-1,k+1)$ and $(x+1,k+1)$, independently of the other sites, where $\lim_{p \to 1}\kappa(p)=1$. In turn, such a process stochastically dominates a usual directed percolation process, in which, $(x,k)$ has a probability $\kappa(p)$ of being connected to  $(x-1,k+1)$, and, independently, a probability $\kappa(p)$ of being connected to  $(x+1,k+1)$. Using the same argument as in \cite{BraNeu} (see \cite{DurNeu, Dur} for the relevant results on oriented percolation), we deduce that, choosing an initial condition $\xi,\zeta$ such that $(0,0)$ is good, 
\begin{equation}\label{e:perco-orient} \liminf_{n \to +\infty} \P_{\xi,\zeta}( (0,2n) \mbox{ is good}   ) > 0.\end{equation}
 Now, let $$E_{x,t} =\left\{ \exists \ \xi_1, \xi_2 \in \{0,1\}^{\Z} \mbox{ such that } X_t^{\xi_1}(x) \neq X_t^{\xi_2}(x)  \right\}.$$ 
By definition (condition (i)), we have that $\{ (z,k) \mbox{ is good }\} \subset \bigcup_{x \in zK+I} E_{x,kT}$, so that, by the union bound, 
$\P_{\xi,\zeta}((z,k) \mbox{ is good }) \leq \sum_{x \in zK+I} \P(E_{x,kT})$. 
By translation invariance, $\P(E_{x,kT})$ does not depend on $x$, so that we have the bound
$\P_{\xi,\zeta}((z,k) \mbox{ is good }) \leq  (2K+1) \P(E_{0,kT})$. In view of \eqref{e:perco-orient}, we deduce that 
\begin{equation}\label{e:la-fin} \liminf_{n \to +\infty} \P(E_{0,2nT})  \geq \frac{1}{2K+1} \liminf_{n \to +\infty} \P_{\xi,\zeta}( (0,2n) \mbox{ is good}   ) > 0,\end{equation}
which yields the conclusion of Theorem \ref{t:letheo}, since $\limsup_{t \to +\infty} \P(E_{0,t}) \geq  \liminf_{n \to +\infty} \P(E_{0,2nT})$.

\bibliographystyle{siam}
\bibliography{no-CFTP-PCA}

\end{document}